\documentclass[12pt]{article}
\usepackage{times}
\usepackage{amsmath,amssymb}
\usepackage{theoremref}
\usepackage[utf8]{inputenc}
\usepackage{indentfirst}
\usepackage{graphicx,tikz}
\usepackage{xcolor}
\usepackage{xskak}
\usepackage[T1]{fontenc}
\usepackage{amsthm}
\usepackage{indentfirst}

\usepackage{parskip}
\newtheorem{problem}{Question}

\newtheorem{lemma}[problem]{Lemma}

\title{Three-Layer Intersecting Temperate Families}
\author{Kada Williams}

\begin{document}

\maketitle

\begin{abstract}
In response to a conjecture of Petr and Turek, we prove that every largest intersecting temperate collection of sets $\mathcal{A}\subseteq [2k]^{(k)}\cup [2k]^{(k+1)}\cup [2k]^{(k+2)}$ includes all $k$-sets containing some given $a\in[2k]$, all $(k+1)$-sets, and all $(k+2)$-sets not containing $a$.
\end{abstract}

\pagenumbering{arabic}
\setlength{\baselineskip}{18pt}

\section{Introduction}

Let $[n]=\{1,2,\ldots,n\}$. A collection $\mathcal{A}$ of subsets of $[n]$ is called {\it temperate} if for all $A\in\mathcal{A}$, there are at most $|A|$ proper subsets of $A$ in $\mathcal{A}$. It is intersecting if for all $A,B\in \mathcal{A}$, the intersection $A\cap B$ is nonempty.

The notion of temperate collection is distantly related to the Sunflower Conjecture. This conjecture states that for every $r\ge 3$, if $\mathcal{C}$ is an $r$-sunflower-free collection of at-most-$m$-element sets, then $|\mathcal{C}|$ is at most exponential in $m$. Here, an {\it $r$-sunflower} is a collection of $r$ sets that contain every element $0$, $1$, or $r$ times \cite{PeT}. Now if $\mathcal{C}$ does not contain an {\it odd sunflower}, a collection whose sets contain every element $0$ or an odd number of times, then $\mathcal{C}$ is $3$-sunflower-free and also temperate. This holds since for all $A\in \mathcal{C}$, at most $|A|$ vectors can be linearly independent in $\mathbb{F}_2^A$.

Petr and Turek asked how large an intersecting temperate collection of subsets of $[n]$ can be. They conjectured the maximum is $\binom{2k-1}{k}+\binom{2k-1}{k+1}$ if $n=2k-1$ and $\frac12\binom{2k}{k}+\binom{2k}{k+1}+\binom{2k-1}{k+2}$ if $n=2k$ ($k\ge 2$) \cite{PeT}. We partially resolve this conjecture.

\section{Complementary Pairs}

The following result was known previously \cite{PeT}. For completeness, we prove it.

\begin{lemma}\label{chains}
    Let $\mathcal{A}$ be a temperate collection of subsets of $[n]$ with $\emptyset\notin\mathcal{A}$. Then
    $$\sum_{A\in \mathcal{A}}\frac1{\binom n{|A|}}\le 2$$
\end{lemma}

\begin{proof} Consider the chain
$$\mathcal{L}=\{\emptyset,\{a_1\},\{a_1,a_2\},\ldots,\{a_1,\ldots,a_n\}\},$$ 
where $a_1,\ldots,a_n$ is a uniformly random permutation of $[n]$. Then, if $A\subsetneq M$ are sets in $\mathcal{A}$, the chance that $A\in \mathcal{L}$ conditional on $M\in \mathcal{L}$ is $\binom{|M|}{|A|}^{-1}\le \frac1{|M|}$. 

Since $\mathcal{A}$ is temperate, there are at most $|M|$ such sets $A$, so their expected number is at most $1$. By partitioning the probability space according to the largest set $M\in\mathcal{A}$ that appears in $\mathcal{L}$ (if exists), it follows that $\mathbb{E}|\mathcal{A}\cap \mathcal{L}|\le 2$.

Since the chance of $A\in \mathcal{L}$ is $\binom{n}{|A|}^{-1}$, this proves the result.
\end{proof}

A \textit{complementary pair} in $[n]$ is $\{A,B\}$, where $A\cap B=\emptyset$ and $A\cup B=[n]$. If $\mathcal{A}$ is intersecting, then it contains no complementary pair as a subset.

Denote $[n]^{(i)}=\{A\subseteq [n]:|A|=i\}$. Now $[n]^{(i)}\cup [n]^{(n-i)}$ contains the complement of each set included, so it can be partitioned into complementary pairs. Hence, an intersecting collection of sets contains at most half the sets.

Let $\mathcal{A}$ be an intersecting temperate collection of subsets of $[n]$. For $n=2k-1$, $\mathcal{A}$ contains at most half the sets in $[2k-1]^{(k)}\cup [2k-1]^{(k-1)}$ and at most half the sets in $[2k-1]^{(k+1)}\cup [2k-1]^{(k-2)}$. This condition suffices to deduce from Lemma \ref{chains} that $[2k-1]^{(k)}\cup [2k-1]^{(k+1)}$ is the largest intersecting temperate collection \cite{PeT}. For $n=2k$, $\mathcal{A}$ contains no complementary pair, and we will show that if $\mathcal{A}\subseteq [2k]^{(k)}\cup [2k]^{(k+1)}\cup [2k]^{(k+2)}$, the largest possible temperate $\mathcal{A}$ is a \textit{lightning}.

\section{Reducing the Conditions}

Let us bear in mind our optimal example, the so-called lightning \cite{PeT}. For some $a\in [2k]$, it includes every $k$-set containing $a$, every $(k+1)$-set, and every $(k+2)$-set not containing $a$. Although there are other optimal examples in general \cite{PeT}, we will show this is the only possible maximiser if $\mathcal{A}\subseteq [2k]^{(k)}\cup [2k]^{(k+1)}\cup [2k]^{(k+2)}$. 

Let $\mathcal{A}\subseteq [2k]^{(k)}\cup [2k]^{(k+1)}\cup [2k]^{(k+2)}$ be a temperate family that contains at most half the sets in $[2k]^{(k)}$. Each $(k+2)$-set contains $k+2$ possible $(k+1)$-sets, and so every $A\in \mathcal{A}\cap [2k]^{(k+2)}$ has at most as many $k$-subsets in $\mathcal{A}$ as many $(k+1)$-subsets it has not in $A$. Let's add up these bounds!
$$\sum_{A\in \mathcal{A},|A|=k+2}m_A\le \sum_{A\in \mathcal{A},|A|=k+2}\quad \sum_{B\notin\mathcal{A},|B|=k+1} 1_{B\subset A},$$
where $m_A$ is the number of subsets $C\in\mathcal{A}$ of $A$ with $|C|=k$. Observe that if $|B|=k+1$, at most $k-1$ $(k+2)$-sets $A$ contain $B$. Exchanging the summation:
$$\sum_{A\in\mathcal{A},|A|=k+2}m_A\le (k-1)\big|[2k]^{(k+1)}\setminus \mathcal{A}\big|.$$
It suffices to prove that $|[2k]^{(k+1)}\setminus \mathcal{A}|\ge |\mathcal{A}\cap [2k]^{(k+2)}|-\binom{2k-1}{k+2}$. Hence, we show
$$\frac1{k-1}\sum_{A\in\mathcal{A},|A|=k+2}m_A\ge \big|\mathcal{A}\cap [2k]^{(k+2)}\big|-\binom{2k-1}{k+2},$$
$$\sum_{A\in\mathcal{A},|A|=k+2}\left(1-\frac{m_A}{k-1}\right)\le \binom{2k-1}{k+2}.$$
In the case of a lightning, $m_A=0$ for all $A\in \mathcal{A}$ of size $k+2$. However, if $A\notin \mathcal{A}$, we observe $m_A=\binom{k+1}{2}$. Hence, it makes good sense to apply the bound $(1-x)\le (1-\frac{x}{y})^2$ with $y=\frac1{k-1}\binom{k+1}{2}\ge 2$ and $x=\frac{d_A}{k-1}$, reducing the claim to
$$\sum_{|A|=k+2}\left(1-\frac{m_A}{\binom{k+1}{2}}\right)^2\le \binom{2k-1}{k+2}.$$

\section{Spectral Decomposition of Hypercube Layers}

Let $V_k$ be a $\mathbb{Q}$-vector space of dimension $\binom nk$ whose orthonormal basis vectors $\mathbf{e}_A^{(k)}$ are indexed by the sets $A\in [n]^{(k)}$. Let $i_k:V_k\to V_k$ be the identity map.

Let $u_k:V_{k-1}\to V_k$ be the linear map defined by $\mathbf{e}_A^{(k-1)}\mapsto \sum_{A\subset B}\mathbf{e}_B^{(k)}$.

Let $d_k:V_k\to V_{k-1}$ be the linear map defined by $\mathbf{e}_A^{(k)}\mapsto \sum_{B\subset A}\mathbf{e}_B^{(k-1)}$.

Clearly, the matrix $D_k$ of $d_k$ is the transpose of the matrix $U_k$ of $u_k$. So if $\mathbf{x}\in V_k$, 
$$\langle d_k\mathbf{x},d_k\mathbf{x}\rangle=\mathbf{x}^TD_k^TD_k\mathbf{x}=\mathbf{x}^TU_kD_k\mathbf{x}=\langle \mathbf{x},u_kd_k\mathbf{x}\rangle.$$

\begin{lemma} \label{harmonic}
    $d_{k+1}\circ u_{k+1}-u_k\circ d_k=(n-2k)i_k$
\end{lemma}

\begin{proof}
We compare $u_k\circ d_k$ with $d_{k+1}\circ u_{k+1}$. When mapping $\mathbf{e}_A^{(k)}$, the former outputs $k\mathbf{e}_A^{(k)}$ plus $\mathbf{e}_B^{(k)}$ for each set $B$ that differs from $A$ by one element, while the latter outputs $(n-k)\mathbf{e}_A^{(k)}$ plus the same. Subtracting yields the result.
\end{proof}

Letting $u$ (or $d$, $i$) be the union of ordered pairs in $u_k$ (or $d_k$, $i_k$) over $1\le k\le n$, it follows that for all $\mathbf{x}\in V_k$,
$$\langle u\mathbf{x},u\mathbf{x}\rangle -\langle d\mathbf{x},d\mathbf{x}\rangle=(n-2k)\langle \mathbf{x},\mathbf{x}\rangle.$$
Therefore, if $1\le j\le \frac n2$, then $d_j$ is surjective, i.e. the rank of $D_j$ is $\binom n{j-1}$, or $u_j$ is injective. Indeed, if $u_j\mathbf{x}=0$ and $\mathbf{x}\neq 0$, then absurdly, $\langle d\mathbf{x},d\mathbf{x}\rangle <0$. Let $W_j\le V_j$ be its kernel, with dimension $\binom nj-\binom{n}{j-1}$ by rank-nullity, and $W_0=V_0$.

\begin{lemma} \label{commute}
    Let $\mathbf{x}\in W_j$ and $1\le l\le n-j$. Then
    $$du^l\mathbf{x}=l(n-2j-l+1)u^{l-1}\mathbf{x}.$$
\end{lemma}

\begin{proof}
    We induct, using Lemma \ref{harmonic}. For $l=1$, use $d\mathbf{x}=0$. For the induction step:
    $$du^{l+1}\mathbf{x}=(du)u^l\mathbf{x}=(ud+(n-2j-2l)i)u^l\mathbf{x}$$
    and $l(n-2j)-l(l-1)+(n-2j-2l)=(l+1)(n-2j)-(l+1)l$.
\end{proof}

\begin{lemma}
    Let $k\le \frac n2$. Then $V_k=\bigoplus_{j=0}^k u^{k-j}W_j$ is an orthogonal decomposition.
\end{lemma}

\begin{proof}
    Let $l<j$, $\mathbf{x}\in W_l$, $\mathbf{y}\in W_j$. Then $\langle u^{k-l}\mathbf{x},u^{k-j}\mathbf{y}\rangle=\langle \mathbf{x},d^{k-l}u^{k-j}\mathbf{y}\rangle$. However, by Lemma \ref{commute} applied $k-j$ times, $d^{k-j}u^{k-j}\mathbf{y}$ is a scalar multiple of $\mathbf{y}$. Calling $d$ one more time annihilates it. Hence, the spaces are orthogonal, and their total dimension is $\binom nk$, by telescoping.
\end{proof}

Recalling the problem at hand, with $n=2k$, let $\mathbf{v}\in V_k$ be the indicator of sets in $\mathcal{A}$. We see that $uu\mathbf{v}=\sum_{|A|=k+2}2m_A\mathbf{e}_A^{(n)}$, because for each $k$-set $C\in\mathcal{A}$ with $C\subset A$, there are two distinct $(k+1)$-sets $B$ sandwiched between them. Hence, 
$$4\sum_{|A|=k+2} m_A^2=\langle u^2\mathbf{v},u^2\mathbf{v}\rangle=\langle \mathbf{v},d^2u^2\mathbf{v}\rangle.$$
Since $\sum_{|A|=k+2}m_A=\binom{k}{2}|\mathcal{A}\cap [2k]^{(k)}|$, we may focus on upper bounding $\langle \mathbf{v},d^2u^2\mathbf{v}\rangle$, or after taking complements, $\langle \mathbf{v},u^2d^2\mathbf{v}\rangle$.

\begin{lemma} \label{spectrum}
    The spectral decomposition of $u^2 d^2$ on $V_k$ is $\bigoplus_{j=0}^k u^{k-j}W_j$. On $W_{k-j}$, the map $u^2d^2$ scales by $(j+2)(j+1)j(j-1)=24\binom{j+2}{4}$.
\end{lemma}

\begin{proof}
    Let $\mathbf{x}\in W_j$. Using Lemma \ref{commute} with $l=k-j$, we see $n-2j-l=k-j$,
    $$u^2d^2u^{k-j}\mathbf{x}=(k-j)(k-j+1)u^2du^{k-j-1}\mathbf{x},$$
    and then using Lemma \ref{commute} again,
    $$u^2d^2u^{k-j}\mathbf{x}=(k-j-1)(k-j+2)(k-j)(k-j+1)u^{k-j}\mathbf{x}.$$
    Therefore, on $W_j$, the linear map $u^2d^2$ is the scalar in this equation.
\end{proof}

Let $\alpha$ be the proportion of sets in $[2k]^{(k)}$ present in $\mathcal{A}$. Then the projection of $\mathbf{v}$ onto $u^kW_0$ is the vector with all components equal to $\alpha$. Since $W_1$ is generated by vectors $e_{\{a\}}^{(1)}-e_{\{1\}}^{(1)}$ ($1<a\le 2k$), the vectors in $u^{k-1}W_1$ have opposite components at complementary pairs. Thus, we can bound $\langle \mathbf{v},u^2d^2\mathbf{v}\rangle$ depending on $\alpha$.

\section{Completing the Proof}

Let $\mathbf{w}=\mathbf{v}-\alpha\mathbf{1}$ be the balanced vector and $\mathbf{w}'$ its projection $\mathbf{w}$ onto the space generated by vectors $\mathbf{z}_A=\mathbf{e}_A^{(k)}-\mathbf{e}_{\overline{A}}^{(k)}$. Since $\mathbf{z}_A\perp \mathbf{1}$, $\langle \mathbf{w},\mathbf{z}_A\rangle=1_{A\in \mathcal{A}}+1_{\overline{A}\in \mathcal{A}}$ and $\langle \mathbf{w},\mathbf{w}\rangle=(\alpha-\alpha^2)\binom{2k}{k}$. Hence, 
$$\langle \mathbf{w}',\mathbf{w}'\rangle = \sum_{1\in A}\frac{\langle \mathbf{w},\mathbf{z}_A\rangle^2}{\langle \mathbf{z}_A,\mathbf{z}_A\rangle}=\frac12\alpha\binom{2k}{k},$$
and by projecting $\mathbf{w}$ onto $u^{k-1}W_1$ in particular, the square length can but decrease. Therefore, using the spectrum found in Lemma \ref{spectrum},
$$\frac{\langle \mathbf{v},u^2d^2\mathbf{v}\rangle}{\binom{2k}{k}}\le \alpha^2\cdot 24\binom{k+2}{4}+\frac12\alpha\cdot 24\binom{k+1}{4}+\left(\frac12\alpha-\alpha^2\right)\cdot 24\binom k4.$$
Recalling that $\sum m_A=\alpha k(k-1)\binom{2k}{k}$, we have
$$\sum_{|A|=k+2}\left(1-\frac{m_A}{\binom{k+1}{2}}\right)^2= \binom{2k}{k+2}-\frac{2\binom{k}{2}\alpha\binom{2k}{k}}{\binom{k+1}{2}}+\frac1{4\binom{k+1}{2}^2}\langle \mathbf{v},u^2d^2\mathbf{v}\rangle.$$
After algebraic manipulations, we deduce that $\frac1{\binom{2k-1}{k+2}}\sum_{|A|=k+2}\left(1-\frac{m_A}{\binom{k+1}{2}}\right)^2\le $
$$\le \frac{8(k+2)(2k-1)\alpha^2-2(k+2)(k^2+5k-2)\alpha+2k^2(k+1)}{k(k+1)(k-2)}\le 1.$$
On the interval $\alpha\in[0,\frac12]$, this quadratic attains the maximum $1$ exactly at $\alpha=\frac12$.

Therefore, if $\mathcal{A}\subseteq [2k]^{(k)}\cup [2k]^{(k+1)}\cup [2k]^{(k+2)}$, then $|\mathcal{A}|\le \frac12\binom{2k}{k}+\binom{2k}{k+1}+\binom{2k-1}{k+2}$.

For equality to hold, $m_A=0$ for all $A\in [2k]^{(k+2)}$ and every $(k+2)$-set in $\mathcal{A}$ has $k+2$ subsets in $\mathcal{A}$. So a $(k+2)$-set in $\mathcal{A}$ cannot contain a $k$-set in $\mathcal{A}$. Hence, by the Kruskal-Katona theorem \cite{Lov}, if the number of $(k+2)$-sets in $\mathcal{A}$ is at least $\binom{2k-1}{k+2}$, then there are at least $\binom{2k-1}{k}=\frac12\binom{2k}k$ $k$-sets not in $\mathcal{A}$. For equality to hold, up to a permutation of $[2k]$, the $(k+2)$-sets are formed of the first $2k-1$ elements. Thus, only a lightning achieves equality.

\section{Acknowledgements and Declaration}

The author would hereby like to thank Jan Petr for introducing this problem to them. Most of the proof ideas were found using a free one-month trial of ChatGPT 5.6 Plus. Thus, the author is grateful for, if not indebted to, anyone whose work the LLM was trained on so infinitesimally as to deprive them of a due citation. In that regard, we dedicate this research to the advancement of mathematical knowledge, in general, and combinatorics, in particular, that it be a source of happiness for all.

\textsc{Independent author in Szeged, Hungary.} \\ \\
\textit{E-mail address:} \texttt{kkw25@cantab.ac.uk}

\end{document}